\documentclass[11pt]{amsart}

\usepackage {enumerate}
\usepackage[english]{babel}
\usepackage{setspace}
\usepackage[hidelinks]{hyperref}
\usepackage{caption}
\usepackage{appendix}
\usepackage{amsrefs}
\usepackage{hyperref}
\usepackage{stmaryrd}
\usepackage{stringstrings}
\allowdisplaybreaks

\newtheorem{theorem}{Theorem}
\newtheorem{lemma}[theorem]{Lemma}

\theoremstyle {definition}
\newtheorem{remark}[theorem]{Remark}
\newtheorem{conjecture}[theorem]{Conjecture}

\usepackage{tikz}
\usetikzlibrary{decorations.pathreplacing}

\usepackage[margin=1in]{geometry}

\allowdisplaybreaks

\newcommand{\R}{\mathbb{R}}

\newcommand{\cM}{\mathcal{M}}
\newcommand{\RR}{\mathbb{R}}

\newcommand{\eps}{\varepsilon}

\BibSpec{article}{%
	+{}{\PrintAuthors} {author}
	+{,}{ } {title}
	+{, }{\textit } {journal}
	+{}{ \parenthesize} {date}
	+{,  }{no. } {volume}
	+{,}{ } {pages}
	+{,}{ } {note}
}

\ExplSyntaxOn

\ExplSyntaxOff
\BibSpec{book}{%
	+{}{\PrintAuthors}  {author}
	+{. }{}{title}
	+{,}{ }{series}
	+{,}{ vol.~}{volume}
	+{, }{\textit}{publisher}
			+{, }{}{date}
	+{. }{ } {note}
}
\BibSpec{collection.article}{%
	+{}{\PrintAuthors}{author}
	+{, }{}{title}
	+{, }{\textit}{booktitle}
	+{, }{ \DashPages}{pages}
	+{,}{ }{series}
	+{, }{}{volume}
	+{, }{\textit}{publisher}
	+{,}{ }{date}
}

\title{On the Minkowski inequality near the sphere}

\author{Otis Chodosh}
\address{\textnormal{Otis Chodosh  \newline \indent
		Stanford University \newline \indent
		Department of Mathematics  \newline \indent
		Building 380  \newline \indent Stanford, CA 94305, USA \newline\indent 
		\href{https://orcid.org/0000-0002-6124-7889}{https://orcid.org/0000-0002-6124-7889} \newline\indent	
		\href{mailto:ochodosh@stanford.edu}{ochodosh@stanford.edu}}
}
\author{Michael Eichmair}
\address{
	\textnormal{Michael Eichmair \newline  \indent
		University of Vienna \newline \indent
		Faculty of Mathematics  \newline \indent
		Oskar-Morgenstern-Platz 1 \newline \indent
		1090 Vienna, 	Austria  \newline\indent 
		\href{https://orcid.org/0000-0001-7993-9536}{https://orcid.org/0000-0001-7993-9536} \newline\indent	
		\href{mailto:michael.eichmair@univie.ac.at}{michael.eichmair@univie.ac.at}}
}

\author{Thomas Koerber}
\address{\textnormal{Thomas Koerber  \newline \indent
		University of Vienna \newline \indent
		Faculty of Mathematics  \newline \indent
		Oskar-Morgenstern-Platz 1 \newline \indent 1090 Vienna,	Austria \newline\indent 
		\href{https://orcid.org/0000-0003-1676-0824}{https://orcid.org/0000-0003-1676-0824} \newline \indent
		\href{mailto:thomas.koerber@univie.ac.at}{thomas.koerber@univie.ac.at}}
}

\begin{document}

\date{\today}
\begin{abstract}
We construct a sequence $\{\Sigma_\ell\}_{\ell=1}^\infty$ of closed, axially symmetric surfaces $\Sigma_\ell\subset \mathbb{R}^3$ that converges to the unit sphere in $W^{2,p}\cap C^1$ for every $p\in[1,\infty)$ and such that, for every $\ell$,
$$
\int_{\Sigma_{\ell}}H_{\Sigma_\ell}-\sqrt{16\,\pi\,|\Sigma_{\ell}|}<0
$$
where $H_{\Sigma_\ell}$ is the mean curvature of $\Sigma_\ell$.
 This shows that unless additional convexity assumptions are imposed, the Minkowski inequality with optimal constant fails even for perturbations  of a round sphere that are small in $W^{2,p}\cap C^1$.
\end{abstract}
\maketitle
\section{Introduction}
Let $\Sigma\subset\R^3$ be a  closed surface and
\[
\cM(\Sigma) = \int_\Sigma H - \sqrt{16\,\pi\, |\Sigma|} 
\]
where $|\Sigma|$ is the area of $\Sigma$ and  $H\in C^\infty(\Sigma)$ is the mean curvature computed with respect to the outward unit normal $\nu$. We denote the second fundamental form of $\Sigma$ with respect to $\nu$ by $h$ and the traceless second fundamental form of $\Sigma$ by $\mathring{h}$.    \\ \indent 
Let $S=\{x\in \mathbb{R}^3:|x|=1\}$ be the unit sphere centered at the origin. Given $u\in C^\infty(S)$, we denote by 
$$
\Sigma(u)=\{(1+u(x))\,x:x\in S\}
$$
 the normal graph of $u$ above $S$.  \\ \indent 
 Note that $\cM(S)=0$. A classical result of H.~Minkowski \cite[\S7]{Minkowski} asserts that, provided that $\Sigma$ is convex, 
 \begin{align} \label{minkowski inequality}\cM(\Sigma)\geq 0\text{ with equality if and only if } \Sigma\text{ is a round sphere.} \end{align}    
P.~Guan and J.~Li \cite[Theorem 2]{GuanLi} have proven that \eqref{minkowski inequality} holds provided that $\Sigma$ is star-shaped and mean-convex. G.~Huisken has found yet not published a proof of \eqref{minkowski inequality} based on inverse mean curvature flow in the case where $\Sigma$ is outward-minimizing; see, e.g., \cite[Theorem 6]{GuanLi}. A.~Freire and F.~Schwartz have independently found a different proof of this, also based on inverse mean curvature flow; see \cite[Theorem 5 and the remark on p.~119]{FreireSchwartz}.  J.~Dalphin, A.~Henrot, S.~Masnou, and T.~Takahashi \cite[Theorem 1.1]{DalphinHenrotMasnouTakahashi} have established \eqref{minkowski inequality} in the case where $\Sigma$ is axially symmetric and such that $\Sigma\cap\Pi$ is connected for every affine plane $\Pi$ orthogonal to the axis of symmetry.\\
\indent By contrast, there are closed surfaces $\Sigma\subset \mathbb{R}^3$ with $\mathcal{M}(\Sigma)<0$. The following examples of such surfaces have been constructed in \cite[Theorem 1.2]{DalphinHenrotMasnouTakahashi}. 
\begin{theorem}\label{DHMT}
There is a sequence $\{\Sigma_\ell\}_{\ell=1}^\infty$ of closed surfaces $\Sigma_\ell\subset\mathbb{R}^3$ with $|\Sigma_\ell|=4\,\pi$ and $\mathcal{M}(\Sigma_\ell)\to-\infty$. There is a sequence $\{\tilde \Sigma_\ell\}_{\ell=1}^\infty$ of  closed axially symmetric surfaces $\tilde \Sigma_\ell\subset\mathbb{R}^3$ with $|\tilde \Sigma_\ell|=4\,\pi$ and $\mathcal{M}(\tilde \Sigma_\ell)\to-8\,\pi$.
\end{theorem}
\indent It is natural to ask whether the Minkowski inequality \eqref{minkowski inequality} holds for surfaces in a neighborhood of $S$. Since every $C^2$-perturbation of the round sphere is strictly convex, we consider lower regularity perturbations. In this direction, F.~Glaudo \cite[Appendix A]{Glaudo} has recently proven the following result.
\begin{theorem} \label{glaudo} There is a sequence $\{u_\ell\}_{\ell=1}^\infty$ of functions $u_\ell\in C^\infty(S)$ with the following properties.
	\begin{itemize}
		\item[$\circ$] $\mathcal{M}(\Sigma(u_\ell))\to-\infty$ %\leq -1$ for every $\ell$
			\item[$\circ$] $|u_\ell|_{C^1(S)}\to0$
	\end{itemize}
	\end{theorem} 
The goal of this paper is to establish the following related result.
\begin{theorem} \label{main result} There is a sequence $\{u_\ell\}_{\ell=1}^\infty$ of axially symmetric functions $u_\ell\in C^\infty(S)$ with the following properties. 
	\begin{itemize}
	\item[$\circ$] $\cM(\Sigma(u_\ell))<0\text{ for every }\ell$ %\leq -1$ for every $\ell$
	\item[$\circ$] $|u_\ell|_{W^{2,p}(S)}\to0$ for every $p\in[1,\infty)$
\end{itemize}

\end{theorem}
\begin{remark} 
Note that, in particular,
	\begin{align*}
		&\circ \qquad |u_\ell|_{C^1(S)}\to0 \text{ and} \qquad \qquad\qquad\qquad\qquad\qquad\qquad\qquad\qquad\qquad\qquad\qquad\qquad\\[-2pt]
&\circ\qquad \int_{\Sigma(u_\ell)}|\mathring{h}|^2\to 0.
	\end{align*}
\end{remark} 

Our interest in pursuing Theorem \ref{main result} stems from the following result due to the first-named and the second-named author; see \cite[Proposition E.4]{angstnomore}.
\begin{theorem}\label{W22 small almost Mink}
Let $\{u_\ell\}_{\ell=1}^\infty$ be a sequence   of functions $u_\ell \in C^\infty(S)$ with $|u_\ell|_{C^1(S)} \to0$. Then, as $\ell\to\infty$, 
\[
\cM(\Sigma(u_\ell)) \geq - o(1) \int_{\Sigma(u_\ell)} |\mathring{h}|^2 .
\]
\end{theorem}
C.\ De Lellis and S.\ M\"uller \cite[Theorem 1.1]{DeLellisMuller} have proven the following quantitative Schur estimate. \begin{theorem} There exists a constant $c\geq 2$ such that, for every closed surface $\Sigma\subset \mathbb{R}^3$,
\begin{align} \label{qs schur estimate}
	\int_{\Sigma}\left(H-|\Sigma|^{-1}\,\int_{\Sigma} H\right)^2\leq c\,\int_{\Sigma}|\mathring{h}|^2.
\end{align}
\end{theorem} 
Note that, if $\Sigma$ has genus zero, then
$$
\int_{\Sigma}\left(H-|\Sigma|^{-1}\,\int_{\Sigma} H\right)^2=2\,\int_\Sigma|\mathring{h}|^2-\sqrt{64\,\pi}\,|\Sigma|^{-1/2}\,\mathcal{M}(\Sigma)-|\Sigma|^{-1}\mathcal{M}(\Sigma)^2
$$
by the Gauss-Bonnet theorem. Consequently, given $\varepsilon>0$, Theorem \ref{W22 small almost Mink} shows that  \eqref{qs schur estimate} holds with $c=2+\varepsilon$ provided that $\Sigma$ is sufficiently close to $S$ in $W^{2,2}\cap C^1$. By contrast, in view of Theorem \ref{main result}, \eqref{qs schur estimate} is false with $c=2$, even for closed surfaces $\Sigma\subset\mathbb{R}^3$ that are $W^{2,p}\cap C^1$-close to $S$ where $p\in[1,\infty)$. See also \cite[Chapter 3]{Perez:thesis} and \cite{Glaudo} for other, related considerations. 

\begin{remark}
Theorem \ref{W22 small almost Mink} implies that the surfaces $\Sigma(u_\ell)$ constructed by F.\ Glaudo as described in Theorem \ref{glaudo} satisfy $$\int_{\Sigma(u_\ell)} |\mathring{h}|^2 \to \infty.$$
\end{remark}
Theorem \ref{W22 small almost Mink} has played a pivotal role in the classification of large stable constant mean curvature surfaces  and large area-constrained Willmore surfaces in initial data sets; see \cite[Lemma 4.8]{angstnomore} and also \cite[Lemma 9]{EichmairKoerberMetzgerSchulze}. This has led the authors to hope that there might exist $\eps_0>0$ such that \eqref{minkowski inequality} holds  provided that  
\[
\int_{\Sigma} |\mathring{h}|^2 < \eps_0.
\]
Theorem \ref{main result} dashes this hope. \\ \indent By contrast, S.~Brendle \cite{Brendle} has recently proved the Michael-Simon Sobolev inequality \cite[Theorem 2.1]{MichaelSimon} with optimal constant; see also the alternative proof of S.~Brendle and the second-named author \cite{BrendleEichmair}.   This result implies the following Minkowski inequality with non-optimal constant for mean-convex surfaces.     
\begin{theorem}[{\cite[Corollary 2]{Brendle}}]\label{simon theorem} Let $\Sigma\subset \mathbb{R}^3$ be a closed surface with non-negative mean curvature. There holds
	$$
	\int_{\Sigma} H\geq \sqrt{4\,\pi\,|\Sigma|}.
	$$
\end{theorem}
Let $\Sigma\subset\mathbb{R}^3$ be a closed surface with non-negative mean curvature such that $|\Sigma|=4\,\pi$.  Theorem \ref{simon theorem} contrasts Theorem \ref{DHMT} and Theorem \ref{glaudo} in the sense that $\mathcal{M}(\Sigma)\geq-4\,\pi$.  Moreover, \cite[Theorem 2]{GuanLi} contrasts Theorem \ref{main result} in the sense that $\mathcal{M}(\Sigma)\geq 0$ provided that $\Sigma$ is $C^1$-close to $S_1(0)$. These reasons support the following conjecture; see also \cite[p.~2732]{DalphinHenrotMasnouTakahashi}.
\begin{conjecture} \label{mean-convex conjecture} Let $\Sigma\subset \mathbb{R}^3$ be a closed surface with non-negative mean curvature. There holds $\mathcal{M}(\Sigma)\geq 0$  with equality if and only if $\Sigma$ is a round sphere. 
\end{conjecture}
\begin{remark}
	By \cite[Theorem 1.3]{DalphinHenrotMasnouTakahashi}, Conjecture \ref{mean-convex conjecture} holds if $\Sigma\subset \mathbb{R}^3$ is assumed to be closed, axially symmetric, and with non-negative mean curvature. 
\end{remark} 
\begin{remark} Let $\Sigma\subset \mathbb{R}^3$ be a closed surface with non-negative mean curvature and  $\Omega\subset\mathbb{R}^3$  the compact domain bounded by $\Sigma$. V.~Agostiniani, M.~Fogagnolo, and L.~Mazzieri \cite[Theorem 1.5]{AgostinianiFogagnoloMazzieri} have shown that \begin{align} \label{volumeminkowski} \int_{\Sigma} H\geq (384\,\pi^2\,|\Omega|)^{1/3}
	\end{align}
with equality if and only if $\Sigma$ is a round sphere.    Note that \eqref{volumeminkowski} would follow from Conjecture \ref{mean-convex conjecture} and the isoperimetric inequality.  
\end{remark} 
F.~Glaudo's proof of Theorem \ref{glaudo} is based on adding dents with very negative mean curvature to the sphere. In view of Theorem \ref{W22 small almost Mink}, the traceless second fundamental form of such dents is large in $L^2$. The construction of the examples in Theorem \ref{main result} is necessarily more subtle. \\ \indent  
We now outline the proof of Theorem \ref{main result}. Let $\{u_\ell\}_{\ell=1}^\infty$ be a sequence of functions $u_\ell\in C^\infty(S)$ with zero mean such that $|u_\ell|_{C^1(S)}\to0$. As in \cite[Appendix A]{Glaudo}, our starting point is a Taylor expansion of the functional $\mathcal{M}$ off the unit sphere $S$. Indeed, we have 
\begin{align*}% \label{taylor}  
\mathcal{M}(\Sigma(u_\ell))=-\frac12\,\int_S |\nabla u_\ell|^2\,\Delta u_\ell+O(1)\,|u_\ell|^2_{W^{1,2}(S)}+O(1)\,|u_\ell|^2_{C^1(S)}\,|u_\ell|_{W^{1,2}(S)}\,|u_\ell|_{W^{2,2}(S)};\end{align*}
see Lemma \ref{expansion}.  Given $\alpha\in(0,1)$, we consider 
$$
u_\ell=-\ell^{-2-\alpha}\,\sum_{i=0}^{2^{-5}\,\ell}v_{\ell_i}
$$
where $v_{\ell_i}\in C^\infty(S)$ are spherical harmonics whose frequency is comparable to $\ell$. A straightforward computation shows that  $|u_\ell|_{W^{2,2}(S)}\to 0$. What is more, given $p\in[1,\infty)$, we show that $|u_\ell|_{W^{2,p}(S)}\to 0$ provided that $\alpha>1-2/p$; see Lemma \ref{estimates w2p}. Next, choosing the spherical harmonics $v_{\ell_i}$ carefully, we obtain that
$$
\liminf_{\ell\to\infty}\ell^{1+3\,\alpha}\, \int_S |\nabla u_\ell|^2\,\Delta u_\ell>0
$$
while 
$$
|u_\ell|^2_{W^{1,2}(S)}+|u_\ell|^2_{C^1(S)}\,|u_\ell|_{W^{1,2}(S)}\,|u_\ell|_{W^{2,2}(S)}=O(1)\,\ell^{-2-2\,\alpha};
$$ 
see Lemma \ref{cubic term}. We note that these estimates build on precise quantitative estimates for E.~Wigner's $3j$-symbol; see Appendix \ref{3j appendix}.

\subsection*{Acknowledgments}
Otis Chodosh was supported by a Terman Fellowship. This research was  funded in whole or in part by the
Austrian Science Fund (FWF) [\href{https://www.fwf.ac.at/en/research-radar/10.55776/Y963}{10.55776/Y963}, \href{https://www.fwf.ac.at/en/research-radar/10.55776/M3184}{10.55776/M3184}, \href{https://www.fwf.ac.at/en/research-radar/10.55776/PAT9828924}{10.55776/PAT9828924}, 
\href{https://www.fwf.ac.at/en/research-radar/10.55776/PAT1307525}{10.55776/PAT1307525}].
\section{Graphs over Euclidean spheres}
Let $u\in C^\infty(S)$ and recall that $\Sigma(u)$ is the normal graph of $u$ over $S$.  Let $g(u)$ be the first fundamental form, $\nu(u)$ the outward unit normal, $h(u)$ the second fundamental form,  and $H(u)$  the mean curvature of $\Sigma(u)$. Let $g$ be the first fundamental form of $S$ and  $f(u)=\sqrt{(1+u)^2+|\nabla u|^2}$. 
\begin{lemma} The following identities hold. \label{graph identities} 
	\begin{align*} 
		 g(u)&=(1+u)^2\,g+du\otimes du\qquad\qquad\qquad\qquad\qquad\qquad\qquad\qquad\qquad\\
			g(u)^{-1}&=(1+u)^{-2}\,\left(g^{-1}-f(u)^{-2}\,\nabla u\otimes \nabla u\right)\\ 
		 \nu(u)&=f(u)^{-1}\,((1+u)\,x-\nabla u)\\
				h(u)&=f(u)^{-1}\,\big[(1+u)^2\,g+2\,du\otimes du-(1+u)\,\nabla d u\big]
	\end{align*}
\end{lemma} 
\begin{lemma} \label{expansion} Let $\{u_\ell\}_{\ell=1}^\infty$ be a sequence of functions $u_\ell\in C^\infty(S)$ with $|u_\ell|_{C^1(S)}\to0$  and
	\begin{align*} 
	\int_S u_\ell=0
	\end{align*} 
	for all $\ell$. Then
	$$
	|\Sigma(u_\ell)|=4\,\pi+O(1)\,|u_\ell|^2_{W^{1,2}(S)}
	$$
	and 
	$$
	\int_{\Sigma(u_\ell)}H=8\,\pi+\int_S \nabla du_\ell(\nabla u_\ell,\nabla u_\ell)+O(1)\,|u_\ell|^2_{W^{1,2}(S)}+O(1)\,|u_\ell|^2_{C^1(S)}\,|u_\ell|_{W^{1,2}(S)}\,|u_\ell|_{W^{2,2}(S)}.
	$$
\end{lemma} 
\begin{proof}
	Note that
	$$
	\sqrt{\det g(u_\ell)}=(1+u_\ell)\,f(u_\ell)
	$$
	and
	\begin{align*} 
	H(u_\ell)&=2\,f(u_\ell)^{-1}+f(u_\ell)^{-3}\,|\nabla u_\ell|^2\\&\qquad -(1+u_\ell)^{-1}\,f(u_\ell)^{-1}\,\Delta u_\ell  +(1+u_\ell)^{-1}\,f(u_\ell)^{-3}\,\nabla du_\ell(\nabla u_\ell,\nabla u_\ell ).
	\end{align*} 
Using Taylor's theorem, we obtain
$$
\sqrt{\det g(u_\ell)}=1+2\,u_\ell+O(1)\,(|u_\ell|^2+|\nabla u_\ell|^2)
$$
and
\begin{align*} 
H(u_\ell)\,\sqrt{\det g(u_\ell)}
&=2\,(1+u_\ell)-\Delta u_\ell+\nabla d u_\ell(\nabla u_\ell,\nabla u_\ell)\\&\qquad+O(1)\,u_\ell^2 +O(1)\,|\nabla u_\ell|^2+O(1)\,(|u_\ell|+|\nabla u_\ell|^2)\,|\nabla u_\ell|^2\,|\nabla^2 u_\ell|.
\end{align*}
The assertion follows from these estimates, using that 
$$
\int_S \Delta u_\ell=0
$$
and 
\begin{align*} 
\int_S  (|u_\ell|+|\nabla u_\ell|^2)\,|\nabla u_\ell|^2\,|\nabla^2 u_\ell|&=O(1)\, |u_\ell|^2_{C^1(S)}\,\int_S  |\nabla u_\ell|\,|\nabla^2 u_\ell|
\\&=O(1)\, |u_\ell|^2_{C^1(S)}\,|u_\ell|_{W^{1,2}(S)}\,|u_\ell|_{W^{2,2}(S)}.
\end{align*} 
\end{proof}
\section{Proof of Theorem \ref{main result} }
Let $k\geq 7$ and $\ell=2^k$. \\ \indent Given an integer $0\leq i\leq 2^{-5}\, \ell,$ we let $\ell_i=\ell+4\,i$. Note that \begin{align*} \ell_i=0\,\,\text{mod}\,\, 4\qquad\text{and} \qquad \ell\leq \ell_i\leq\tfrac98\, \ell.\end{align*}  \indent 
Recall the definitions of $\Lambda_{\ell_i}\subset C^\infty(S)$ and $v_{\ell_i}\in \Lambda_{\ell_i}$ from Appendix \ref{spherical harmonics} and the  definition  of E.~Wigner's $3j$-symbol from Appendix \ref{3j appendix}.\\
 \indent 
Let $\alpha\in(0,1)$.  We define $u_\ell\in C^\infty(S)$ by 
$$
u_\ell=-\ell^{-2-\alpha}\sum_{i=0}^{2^{-5}\,\ell}v_{\ell_i}.
$$
\begin{lemma} There holds \label{estimates} $|u_\ell|_{W^{1,2}(S)}=O(1)\,\ell^{-1-\alpha}.$
\end{lemma} 
\begin{proof} 
By Lemma \ref{inner products}, as $\ell\to\infty$,
$$
\int_S |\nabla u_\ell|^2=O(1)\,\ell^{-2-2\,\alpha}
$$
and 
$$
\int_S u_\ell^2=O(1)\,\ell^{-4-2\,\alpha}.
$$
\end{proof} 
\begin{lemma} There holds $|u_\ell|_{C^{1}(S)}=O(1)\,\ell^{-\alpha}$. \label{c1 est}
\end{lemma}
\begin{proof}
	Using Lemma \ref{pointwise estimate 0} and Lemma \ref{pointwise estimate 1}, we obtain
	$$
	|u_\ell|_{C^{1}(S)}=O(1)\,\ell^{-2-\alpha}\sum_{i=0}^{2^{-5}\,\ell}|v_{\ell_i}|_{C^1(S)}=O(1)\,\ell^{-2-\alpha}\sum_{i=0}^{2^{-5}\,\ell}\ell_i=O(1)\,\ell^{-\alpha}.
	$$
\end{proof} 
\begin{lemma} \label{estimates w2p}
	Let $m\geq 2$ be an even integer.	There holds $|u_\ell|_{W^{2,m}(S)}=O(\ell^{1-\alpha-2/m}).$
\end{lemma}
\begin{proof}
	Since $u_\ell$ has zero mean,  
	$$
	|u_\ell|^m_{W^{2,m}(S)}=O(1)\,\int_{S}(\Delta u_\ell)^m
	$$
by elliptic regularity;	see, e.g.,~\cite[Theorem 9.11]{GilbargTrudinger}.\\
\indent Assume that $m=2$. By Lemma \ref{inner products}, 
$$
\int_S (\Delta u_\ell)^2= \ell^{-4-2\,\alpha}\,\sum_{i=0}^{2^{-5}\,\ell} \int_S (\Delta v_{\ell_i})^2 =O(1)\,\ell^{-4-2\,\alpha}\,\sum_{i=0}^{2^{-5}\,\ell}\ell_i^3=O(1)\,\ell^{-2\,\alpha}. 
$$
\indent Assume that $m\geq 4$. 	Using that $v_{\ell_i}\in \Lambda_{\ell_i}$, we obtain that 
	$$
	\int_S (\Delta u_\ell)^m=\ell^{-2\,m-\alpha\,m}\,\sum_{i_1,\ldots, i_m=0}^{2^{-5}\,\ell}\,\int_S\,\prod_{j=1}^m \Delta v_{\ell_{i_j}}=\ell^{-\alpha\,m}\,\sum_{i_1,\ldots,i_m=0}^{2^{-5}\,\ell}\,O(1)\,\bigg|\int_S\,\prod_{j=1}^m v_{\ell_{i_j}}\bigg|.
	$$ 
Recall  the definition of $\Gamma(\ell_{i_1},\ldots, \ell_{i_{m-1}})$ from \eqref{gamma coefficient}. 	By Lemma \ref{m product},
	\begin{align*}
		\int_S\,\prod_{j=1}^m v_{\ell_{i_j}}&=\sum_{ z\in\Gamma(\ell_{i_1},\ldots, \ell_{i_{m-1}})} \delta_{z_{{m-1}}\ell_{i_m}} \,(2\,\ell_{i_m}+1)^{-1}
		\,\prod_{j=1}^{m-2}(2\,z_{j+1}+1)\begin{pmatrix} z_{{j}} & \ell_{i_{j+1}} & z_{j+1} \\
			0 & 0 &0 \end{pmatrix}^2
		\\&=O(1)\,\ell^{-1}\,\sum_{ z\in\Gamma(\ell_{i_1},\ldots, \ell_{i_{m-1}})} \,\delta_{z_{{m-1}}\,\ell_{i_m}}
		\,\prod_{j=1}^{m-2}(2\,z_{j+1}+1)\begin{pmatrix} z_{{j}} & \ell_{i_{j+1}} & z_{j+1}  \\
			0 & 0 &0 \end{pmatrix}^2.
	\end{align*} 
Clearly,
	\begin{align*}
		&\,\sum_{i_m=0}^{2^{-5}\,\ell}\,\sum_{ z\in\Gamma(\ell_{i_1},\ldots, \ell_{i_{m-1}})} \,\delta_{z_{{m-1}}\,\ell_{i_m}}
		\,\prod_{j=1}^{m-2}(2\,z_{j+1}+1)\begin{pmatrix} z_{{j}} & \ell_{i_{j+1}} & z_{j+1}  \\
			0 & 0 &0 \end{pmatrix}^2		
		\\&\qquad\leq \sum_{ z\in\Gamma(\ell_{i_1},\ldots, \ell_{i_{m-1}})} 
		\,\prod_{j=1}^{m-2}(2\,z_{j+1}+1)\begin{pmatrix} z_{{j}} & \ell_{i_{j+1}} & z_{j+1} \\
			0 & 0 &0 \end{pmatrix}^2.
	\end{align*} 
	Applying Lemma \ref{summation formula 2}, we see that 
	$$
	\sum_{z\in\Gamma(\ell_{i_1},\ldots, \ell_{i_{m-1}})} 
	\,\prod_{j=1}^{m-2}(2\,z_{j+1}+1)\begin{pmatrix} z_{{j}} & \ell_{i_{j+1}} & z_{j+1} \\
		0 & 0 &0 \end{pmatrix}^2=O(1).
	$$
	The assertion follows, using that 
	$$
	\sum_{i_1,\ldots,i_{m-1}=0}^{2^{-5}\,\ell}1=O(1)\,\ell^{m-1}.
	$$
\end{proof} 

\begin{lemma}
	There holds \label{cubic term}
	$$
	\liminf_{\ell\to\infty} \ell^{1+3\,\alpha}\,\int_S|\nabla u_\ell|^2\,\Delta u_\ell>0.
	$$
\end{lemma} 
\begin{proof} 
Using  Lemma \ref{triple product 1} and Lemma \ref{m product}, we have 
\begin{align*} 
\frac{1}{4\,\pi}\,\int_S |\nabla u_\ell|^2\,\Delta u_\ell =\frac12\,\ell^{-6-3\,\alpha}\,\sum_{i_1,\,i_2,\,i_3=0}^{2^{-5}\,\ell}\,\ell_{i_1}\,(\ell_{i_1}+1)&\,\big[\ell_{i_2}\,(\ell_{i_2}+1)+\ell_{i_3}\,(\ell_{i_3}+1)-\ell_{i_1}\,(\ell_{i_1}+1)\big]\\&\,\times \begin{pmatrix} \ell_{i_1} & \ell_{i_2} &\ell_{i_3} \\
	0& 0 & 0
\end{pmatrix}^2.
\end{align*} 
Note that
$$
\ell_{i_2}\,(\ell_{i_2}+1)+\ell_{i_3}\,(\ell_{i_3}+1)-\ell_{i_1}\,(\ell_{i_1}+1)\geq \ell^2/2 
$$
so that 
$$
\ell_{i_1}\,(\ell_{i_1}+1)\,\big[\ell_{i_2}\,(\ell_{i_2}+1)+\ell_{i_3}\,(\ell_{i_3}+1)-\ell_{i_1}\,(\ell_{i_1}+1)\big]\geq \ell^4/2.
$$
In conjunction with Lemma \ref{3j lower bound},  as $\ell\to\infty$,
$$
\int_S |\nabla u_\ell|^2\,\Delta u_\ell\geq  \left(2^{-18}+o(1)\right)\,\ell^{-1-3\,\alpha}.
$$ 
\end{proof} 
\begin{proof}[Proof of Theorem \ref{main result}]
	Using Lemma \ref{expansion}, Lemma \ref{estimates}, Lemma \ref{c1 est},  Lemma \ref{estimates w2p}, and integration by parts, we see that
	$$
	\cM(\Sigma(u_\ell))=-\frac12\,\int_S |\nabla u_\ell|^2\,\Delta u_\ell+O(1)\,\ell^{-2-2\,\alpha}+O(1)\,\ell^{-1-4\,\alpha}.
	$$
	By Lemma \ref{cubic term}, the right-hand side is negative for all $\ell$   sufficiently large. 
Let $m\geq 2$ be an even integer. By Lemma \ref{estimates w2p},
	$
	|u_\ell|_{W^{2,m}(S)}\to 0
	$ provided that $\alpha>1-2/m$. Using a diagonal argument, the assertion follows.
\end{proof} 
\begin{appendices}

\section{Spherical harmonics and Legendre polynomials}\label{spherical harmonics}
Recall that the eigenvalues of the operator 
	$$-\Delta:H^{2}(S)\to L^2(S)$$
	are given by
	$
	\ell\,(\ell+1)
	$	where $\ell\geq0$ is an integer. We let $\Lambda_\ell$ denote the eigenspace corresponding to the eigenvalue $\ell\,(\ell+1)$ and recall that $\dim\,\Lambda_\ell=2\,\ell+1$.  \\ \indent The Legendre polynomials $P_\ell$ may be defined via a generating function as follows. Given $s\in[0,1]$ and $t\in[0,1)$, there holds
\begin{align*}
	(1-2\,s\,t+t^2)^{-\frac12}=\sum_{\ell=0}^\infty P_\ell(s)\,t^\ell. %\label{legendre generating2}
\end{align*} 
We define $v_\ell\in C^\infty(S)$ by $$v_\ell(x^1,x^2,x^3)=P_\ell(-x^3). $$ Recall from, e.g.,~\cite[\S8]{CourantHilbert} that $v_\ell\in \Lambda_\ell$, namely, $-\Delta v_{\ell}=\ell\,(\ell+1)\,v_\ell$. 
\begin{lemma}[{\cite[\S8]{CourantHilbert}}] \label{inner products}
Let $\ell_1,\ell_2\geq 0$ be integers. The following identities  hold.
\begin{align*} 
	 		\frac{1}{4\,\pi}\,\int_S v_{\ell_1}\,v_{\ell_2}&=\frac{1}{2\,\ell_1+1}\,\delta_{\ell_1\ell_2}\qquad\qquad\qquad\qquad\qquad\qquad\qquad\\ 
	\frac{1}{4\,\pi}\,\int_S \nabla v_{\ell_1}\cdot \nabla v_{\ell_2}&=\frac{\ell_1\,(\ell_1+1)}{2\,\ell_1+1}\,\delta_{\ell_1\ell_2}\\ 
	\frac{1}{4\,\pi}\,\int_S \Delta v_{\ell_1}\,\Delta v_{\ell_2}&=\frac{\ell_1^2\,(\ell_1+1)^2}{2\,\ell_1+1}\,\delta_{\ell_1\ell_2}
\end{align*} 

\end{lemma}
\begin{lemma}  \label{triple product 1}
	Let $\ell_1,\,\ell_2,\,\ell_3\geq 0$ be integers. There holds
	$$
	\int_S \Delta v_{\ell_1}\,\nabla v_{\ell_2}\cdot \nabla v_{\ell_3}=-
	\frac12\,\ell_1\,(\ell_1+1)\,\big[\ell_2\,(\ell_2+1)+\ell_3\,(\ell_3+1)-\ell_1\,(\ell_1+1)\big]
	\int_S v_{\ell_1}\,v_{\ell_2}\,v_{\ell_3}.
	$$
\end{lemma} 
\begin{proof} We have
	\begin{align*} 
&	\operatorname{div}\,(v_{\ell_1}\,v_{\ell_2}\,\nabla v_{\ell_3}+v_{\ell_1}\,v_{\ell_3}\,\nabla v_{\ell_2}-v_{\ell_2}\,v_{\ell_3}\,\nabla v_{\ell_1})\\&\qquad =2\,v_{\ell_1}\,\nabla v_{\ell_2}\cdot \nabla v_{\ell_3}+v_{\ell_1}\,v_{\ell_2}\,\Delta v_{\ell_3}+v_{\ell_1}\,v_{\ell_3}\,\Delta v_{\ell_2}-v_{\ell_2}\,v_{\ell_3}\,\Delta v_{\ell_1}.
	\end{align*} 
Integrating and using that $v_{\ell_i}\in \Lambda_{\ell_i},\,i=1,\,2,\,3$, the assertion follows.
\end{proof}
The proof of the following well-known lemmas is  based on \cite[Chapter IV, Lemma 2.8 and Corollary 2.9]{SteinWeiss}. %\cite[Proposition 6.0.1]{Garrett}. 
\begin{lemma} \label{pointwise estimate 0}
	There holds, as $\ell\to\infty$,
	$$
	|v_\ell|_{C^0(S)}=O(1).
	$$
\end{lemma}
\begin{proof}
	Given $x \in S$, consider the linear map $\Lambda_\ell \to \RR$ given by $f\mapsto f(x)$. By the Riesz representation theorem, for every $x\in S$, there is a unique $F_x \in \Lambda_\ell$ so that $$f(x) = \int_S f\,F_x. $$ In particular,
	\[
	|f(x)| \leq | f|_{L^2(S)}\,| F_x |_{L^2(S)}. 
	\]
	Since $SO(3)$ acts transitively on $S$, it follows that  $ | F_x|_{L^2(S)}^2$ is independent of $x$.
	
Let $\{f_i\}_{i=1}^{2\,\ell+1}$ be an orthonormal basis of $\Lambda_\ell$, we write
	\[
	F_x = \sum_{i=1}^{2\,\ell+1} \left(\int_S f_i\,F_x\right)\,f_i =  \sum_{i=1}^{2\,\ell+1} f_i(x)\, f_i
	\]
	so that
	\[
	| F_x |_{L^2(S)}^2 =  \sum_{i=1}^{2\,\ell+1} f_i(x)^2.
	\]
	Integrating with respect to $x$, we obtain
	\[
	4\,\pi\, | F_x|_{L^2(S)}^2 = 2\,\ell+1.
	\]
It follows that
	\[
	| f|_{C^0(S)} \leq \sqrt{\frac{2\,\ell+1 }{4\,\pi }}\, | f |_{L^2(S)} =O(1)\,\ell^{1/2}\,|f|_{L^2(S)}
	\]
	for all $f \in \Lambda_\ell$. \\	\indent Using also Lemma \ref{inner products}, the assertion follows. 
	\end{proof}  
	\begin{lemma} \label{pointwise estimate 1}
		There holds, as $\ell\to\infty$,
		$$
		|\nabla v_\ell|_{C^0(S)}=O(1)\,\ell.
		$$
	\end{lemma}
\begin{proof} 
	Given $x \in S$, consider the linear map $\Lambda_\ell \to \RR^3$ given by $f\mapsto (\nabla f)(x)$. By the Riesz representation theorem, there is a unique $$\vec F_x=(F_x^1,\,F_x^2,\,F_x^3) \in \Lambda_\ell\times \Lambda_\ell\times \Lambda_\ell$$ such that, for all $f\in \Lambda_\ell$,
	\[
	\nabla f(x) = \int_Sf\,\vec F_x 
	\]
As in the proof of Lemma \ref{pointwise estimate 0}, we may argue by symmetry that  $| \vec F_x|_{L^2(S)}$ is independent of $x$. 
	
	Given an orthonormal basis $\{f_i\}_{i=1}^{2\,\ell+1}$ of $\Lambda_\ell,$ we write
	\[
	\vec F_x = \sum_{i=1}^{2\,\ell+1}   \left(\int_S f_i\,\vec F_x\right)\,f_i = \sum_{i=1}^{2\,\ell+1} (\nabla f_i)(x)\, f_i.
	\]
	Note that 
	\[
	| \vec F_x|_{L^2(S)}^2 = \sum_{i=1}^{2\,\ell+1} | \nabla f_i(x)|^2 
	\]
	Integrating with respect to $x$, we obtain that
	\[
	4\,\pi \,|\vec F_x|_{L^2(S)}^2 = \sum_{i=1}^{2\,\ell+1} | \nabla f_i|_{L^2(S)}^2 = \ell\,(\ell+1) (2\,\ell+1).  
	\]
	Thus,
	\[
	| \nabla f|_{C^0(S)} \leq \sqrt{\frac{\ell\,(\ell+1)\,(2\,\ell+1)}{4\,\pi}}\, | f |_{L^2(S)} =O(1)\,\ell^{3/2} \,| f |_{L^2(S)} 
	\]
	for every $f \in \Lambda_\ell$. \\
	\indent In conjunction with Lemma \ref{inner products}, the assertion follows. 
	\end{proof} 
\section{The $3j$-symbol} \label{3j appendix}
Recall from Appendix \ref{spherical harmonics} the definition of $v_{\ell}\in C^\infty(S)$ where $\ell\geq0$ is an integer. \\ \indent 
Given two integers $\ell_1,\,\ell_2\geq0$, let 
$$
\Gamma_{\ell_1\ell_2}=\{\ell\in \mathbb{Z}:|\ell_1-\ell_2|\leq  \ell\leq  \ell_1+\ell_2\};
$$
cp.~\cite[\S3.7]{Edmonds2}. Note that $\ell\in \Gamma_{\ell_1\ell_2}$ if and only if it is possible to form a possibly degenerate triangle with side lengths $\ell_1,$ $\ell_2$, and $\ell$. Given $\ell_3\in \Gamma_{\ell_1\ell_2}$, recall from, e.g.,~\cite[\S3.7]{Edmonds2}, E.~Wigner's $3j$-symbol given by
\begin{align*} %\label{wigner odd}
	\begin{pmatrix} \ell_1 & \ell_2 &\ell_3 \\
		0& 0 & 0
	\end{pmatrix}=0
\end{align*}
if  $\ell_1+\ell_2+\ell_3$ is odd and
\begin{equation} \label{wigner even}
	\begin{aligned}  
		\begin{pmatrix} \ell_1 & \ell_2 &\ell_3 \\
			0& 0 & 0
		\end{pmatrix}&=(-1)^{(\ell_1+\ell_2+\ell_3)/2}\,\sqrt{\frac{(\ell_1+\ell_2-\ell_3)!\,(\ell_1+\ell_3-\ell_2)!\,(\ell_2+\ell_3-\ell_1)!}{(\ell_1+\ell_2+\ell_3+1)!}}\, \\
		&\qquad \times  \frac{((\ell_1+\ell_2+\ell_3)/2)!}{((\ell_1+\ell_2-\ell_3)/2)!\,((\ell_1+\ell_3-\ell_2)/2)!\,((\ell_2+\ell_3-\ell_1)/2)!}
	\end{aligned}
\end{equation} 
if  $\ell_1+\ell_2+\ell_3$ is even; see \cite[(3.7.17)]{Edmonds2}.
\begin{lemma}[{\cite[(3.5.7)]{Edmonds2}}%or {\cite[(4.6.5)]{Edmonds}}
	] \label{product}
	Let $\ell_1,\,\ell_2\geq 0$ be integers. There holds
	$$
	v_{\ell_1}\,v_{\ell_2}=\sum_{\ell\in \Gamma_{\ell_1\ell_2}}\,(2\,\ell+1)\begin{pmatrix} \ell_1 & \ell_2 &\ell \\
		0& 0 & 0
	\end{pmatrix}^2\,v_\ell.
	$$	
\end{lemma}
Let $m\geq 3$ be an integer. Given integers $\ell_1,\ldots,\ell_{m-1}\geq0$, we let \begin{align} \label{gamma coefficient} \Gamma(\ell_1,\ldots,\ell_{m-1})=\{(z_1,\ldots,z_{m-1})\in \mathbb{Z}^{m-1}:z_1=\ell_1\text{ and }z_i\in \Gamma_{z_{i-1}\ell_{i}} \text{ for } i=2,\ldots,m-1 \}.\end{align} 
Note that $\Gamma(\ell_1,\ell_2)=\{(\ell_1,\ell):\ell\in \Gamma_{\ell_1\ell_2}\}$.
\begin{lemma} \label{m product}
	Let $\ell_1,\,\ldots,\,\ell_{m}\geq0$ be  integers.  There holds 
	$$
	\frac{1}{4\,\pi}\,\int_S\,\prod_{i=1}^m v_{\ell_i}=\sum_{ z\in\Gamma(\ell_1,\ldots,\ell_{m-1})}\delta_{z_{m-1}\ell_m}\,(2\,\ell_m+1)^{-1}\,\prod_{i=1}^{m-2}(2\,z_{i+1}+1)\begin{pmatrix} z_{i} & \ell_{i+1} & z_{i+1} \\
		0& 0 & 0
	\end{pmatrix}^2.
	$$
	
\end{lemma} 
\begin{proof}
	Note that 
	$$
%	\Gamma(\ell_1,\ldots,\ell_{m-1})=\{(\ell_1,z_2c,\ldots,z_{m-1}):z_2\in \Gamma_{\ell_1\ell_2}\text{ and }(z_2,\ldots,z_{m-1})\in \Gamma(z_2,\ell_3,\ldots,\ell_{m-1})\}.
	$$
	The assertion now follows by induction, using Lemma \ref{product} and Lemma \ref{inner products}.
\end{proof}

\begin{lemma}  \label{3j lower bound} Let $\delta>0$. There is an integer $\ell>1$ with the following property. Let $\ell_1,\,\ell_2\geq0$ be integers and $\ell_3\in \Gamma_{\ell_1\ell_2}$ such that
	\begin{itemize}
		\item[$\circ$] $\ell_1+\ell_2+\ell_3=0\, \operatorname{mod}\, 4$ \text{ and}
		\item[$\circ$] $\ell_1,\,\ell_2,\,\ell_3\geq \ell.$
	\end{itemize}
	Then 
	\begin{align*} 
		(\ell_1+\ell_2+\ell_3)\begin{pmatrix} \ell_1 & \ell_2 &\ell_3 \\
			0& 0 & 0
		\end{pmatrix}\geq \sqrt{\frac{2}{\pi}}-\delta.
	\end{align*} 
\end{lemma} 
\begin{proof}
	We first consider the case where $\ell_1+\ell_2>\ell_3$, $\ell_1+\ell_3>\ell_2$, and $\ell_2+\ell_3>\ell_1$.
	
	Consider the sequence $\{a_k\}_{k=1}^\infty$ given by
	$$
	a_k= \frac{((2\,k)!)^{1/2}}{k!}\,\frac{(2\,k)^{1/4}}{2^k}.
	$$	
	By Stirling's formula, 
	$
	\lim_{k\to\infty}	a_k=(2/\pi)^{1/4}.
	$
	In particular, 
	\begin{align} \label{inf} 
		0<	\inf_{k\geq 1}a_k<\sup_{k\geq 1} a_k<\infty.
	\end{align}
	Using \eqref{wigner even} and that $(-1)^{(\ell_1+\ell_2+\ell_3)/2}=1$, we see that
	\begin{equation}  \label{sequence}
		\begin{aligned} 
			&\begin{pmatrix} \ell_1 & \ell_2 &\ell_3 \\
				0& 0 & 0
			\end{pmatrix} \,\bigg[(\ell_1+\ell_2-\ell_3)\,(\ell_1+\ell_3-\ell_2)\,(\ell_2+\ell_3-\ell_1)\,(\ell_1+\ell_2+\ell_3+1)\,\bigg]^{1/4} \\
			&\qquad =\frac{(\ell_1+\ell_2+\ell_3)^{1/4}}{(\ell_1+\ell_2+\ell_3+1)^{1/4}}\,\frac{a_{(\ell_1+\ell_2-\ell_3)/2}\,a_{(\ell_1+\ell_3-\ell_2)/2}\,a_{(\ell_2+\ell_3-\ell_1)/2}}{a_{(\ell_1+\ell_2+\ell_3)/2}}.
		\end{aligned}
	\end{equation} \indent 
	If, as $\ell\to\infty$, 
	$$
	\ell_1+\ell_2-\ell_3=O(1),\qquad \ell_1+\ell_3-\ell_2=O(1),\qquad\text{or}\qquad \ell_2+\ell_3-\ell_1=O(1),
	$$ then $$
	\liminf_{\ell\to\infty} (\ell_1+\ell_2+\ell_3)\,\big[(\ell_1+\ell_2-\ell_3)\,(\ell_1+\ell_3-\ell_2)\,(\ell_2+\ell_3-\ell_1)\,(\ell_1+\ell_2+\ell_3+1)\,\big]^{-1/4}=\infty.
	$$
	In conjunction with \eqref{inf} and \eqref{sequence}, we see that
	\begin{align*} 
		(\ell_1+\ell_2+\ell_3)\begin{pmatrix} \ell_1 & \ell_2 &\ell_3 \\
			0& 0 & 0
		\end{pmatrix}\to\infty.
	\end{align*} 
	\indent If, as $\ell\to\infty$, 
	$$
	\ell_1+\ell_2-\ell_3\to\infty,\qquad \ell_1+\ell_3-\ell_2\to\infty,\qquad\text{and}\qquad \ell_2+\ell_3-\ell_1\to\infty,
	$$ then 
	$$
	\liminf_{\ell\to\infty} (\ell_1+\ell_2+\ell_3)\,\big[(\ell_1+\ell_2-\ell_3)\,(\ell_1+\ell_3-\ell_2)\,(\ell_2+\ell_3-\ell_1)\,(\ell_1+\ell_2+\ell_3+1)\,\big]^{-1/4}\geq 1
	$$
	and 
	$$
	a_{(\ell_1+\ell_2-\ell_3)/2},\,a_{(\ell_1+\ell_3-\ell_2)/2},\,a_{(\ell_2+\ell_3-\ell_1)/2},\,a_{(\ell_1+\ell_2+\ell_3)/2}\to\left(\frac{2}{\pi}\right)^{1/4}. 
	$$ 
	In conjunction with \eqref{sequence},
	\begin{align*} 
		\liminf_{\ell\to\infty} (\ell_1+\ell_2+\ell_3)\begin{pmatrix} \ell_1 & \ell_2 &\ell_3 \\
			0& 0 & 0
		\end{pmatrix}\geq \sqrt{\frac{2}{\pi}}.
	\end{align*} 
\indent The case where $\ell_1+\ell_2=\ell_3$, $\ell_1+\ell_3=\ell_2$, or $\ell_2+\ell_3=\ell_1$ is similar and only requires formal modifications.

	The assertion follows.
\end{proof}
\begin{lemma} 	\label{3j upper bound}  There holds
	\begin{align*} 
\sup_{\ell_1,\,\ell_2,\,\ell_3}		&\bigg|\begin{pmatrix} \ell_1 & \ell_2 &\ell_3 \\
			0& 0 & 0
		\end{pmatrix}\bigg|\,\big[(\ell_1+\ell_2-\ell_3+1)\,(\ell_1+\ell_3-\ell_2+1)\,(\ell_2+\ell_3-\ell_1+1)\,(\ell_1+\ell_2+\ell_3+1)\big]^{1/4} \\&\qquad <\infty.
	\end{align*} 
where the supremum is taken over all integers $\ell_1,\,\ell_2\geq 0$ and all $\ell_3\in \Gamma_{\ell_1\ell_2}$.
\end{lemma} 
\begin{proof}
	This follows from combining  \eqref{inf} and \eqref{sequence} as in the proof of Lemma \ref{3j lower bound}.
\end{proof}
\begin{lemma} \label{summation formula 2} 
	
	There holds 
	$$
\sup_{\ell_1,\,\ell_2}	\sum_{\ell\in \Gamma_{\ell_1\ell_2}} \ell\begin{pmatrix} \ell_1 & \ell_2 & \ell \\
		0& 0 & 0
	\end{pmatrix}^2<\infty 
	$$
	where the supremum is taken over all integers $\ell_1,\,\ell_2\geq 0$. 
\end{lemma} 
\begin{proof} 
	Assume that $\ell_1\geq \ell_2.$ By Lemma \ref{3j upper bound}, we have
	\begin{align*}
		&\sum_{\ell\in \Gamma_{\ell_1\ell_2}} \ell\begin{pmatrix} \ell_1 & \ell_2 & \ell \\
			0& 0 & 0
		\end{pmatrix}^2\\& \qquad =O(1)\,\sum_{\ell=\ell_1-\ell_2}^{\ell_1+\ell_2}\ell\,\left[(\ell_1+\ell_2-\ell+1)\,(\ell_1+\ell-\ell_2+1)\,(\ell_2+\ell-\ell_1+1)\,(\ell_1+\ell_2+\ell+1)\right]^{-1/2}.
	\end{align*}
	Note that
	$$
	\ell\,(\ell_1+\ell-\ell_2+1)^{-1/2}\,(\ell_1+\ell_2+\ell+1)^{-1/2}
	\leq 1.
	$$
	Moreover, 
	\begin{align*} 
		\sum_{\ell=\ell_1-\ell_2}^{\ell_1+\ell_2}(\ell_1+\ell_2-\ell+1)^{-1/2}\,(\ell_2+\ell-\ell_1+1)^{-1/2}
		&=\sum_{i=0}^{2\,\ell_2}(2\,\ell_2-i+1)^{-1/2}\,(i+1)^{-1/2}
		\\&=O(1)\,\int_{0}^{2\,\ell_2}(2\,\ell_2+1-x)^{-1/2}\,x^{-1/2}\,\mathrm{d}x
		\\&=O(1).
	\end{align*}
	\indent The assertion follows.
\end{proof}

\end{appendices}

%\bibliographystyle{amsplainnat}
%\bibliography{literature} 

\end{document}